\documentclass{amsart}
\usepackage{amssymb,amsmath,amsthm,enumerate,amscd,graphicx,color}
\usepackage[usenames,dvipsnames,svgnames,table]{xcolor}
\usepackage{pdfsync}
 \usepackage[colorlinks=true, pdfstartview=FitV, linkcolor=blue, citecolor=blue, urlcolor=blue]{hyperref}
\newcommand{\End}{\text{\rm End}}
\newcommand{\thmref}[1]{Theorem~\ref{#1}}

\newcommand{\lemref}[1]{Lemma~\ref{#1}}
\newcommand{\eqnref}[1]{~(\ref{#1})}

\newcommand{\germ}{\mathfrak}
\newtheorem{thm}{Theorem}[section]
\newtheorem{lem}[thm]{Lemma}

\theoremstyle{definition}
\newtheorem{cor}[thm]{Corollary}

\newtheorem{prop}[thm]{Proposition}


\theoremstyle{rem}
\numberwithin{equation}{section}

\usepackage{amssymb}
\begin{document}
\title{Realizations of the Three Point Lie Algebra $\mathfrak{sl}(2, \mathcal R)  \oplus\left( \Omega_{\mathcal R}/d{\mathcal R}\right)$.}
\author{Ben Cox}
\begin{abstract}  We describe the universal central extension of the three point current algebra $\mathfrak{sl}(2,\mathcal  R)$ where $\mathcal R=\mathbb C[t,t^{-1},u\,|\,u^2=t^2+4t ]$ and construct realizations of it in terms of sums of partial differential operators.
\end{abstract}
\address{Department of Mathematics \\
The College of Charleston \\
66 George Street  \\
Charleston SC 29424, USA}\email{coxbl@cofc.edu}
\urladdr{http://coxbl.people.cofc.edu/papers/preprints.html}
\author{Elizabeth Jurisich}
\address{Department of Mathematics,
The College of Charleston,
Charleston SC 29424}
\email{jurisiche@cofc.edu}

\maketitle
\section{Introduction}  

It is well known from the work of C. Kassel and J.L. Loday (see \cite{MR694130}, and \cite{MR772062}) that if $R$ is a commutative algebra and $\mathfrak g $ is a simple Lie algebra, both defined over the complex numbers, then the universal central extension $\hat{{\mathfrak g}}$ of ${\mathfrak g}\otimes R$ is the vector space $\left({\mathfrak g}\otimes R\right)\oplus \Omega_R^1/dR$ where $\Omega_R^1/dR$ is the space of K\"ahler differentials modulo exact forms (see \cite{MR772062}).  The vector space $\hat{{\mathfrak g}}$ is made into a Lie algebra by defining
$$
[x\otimes f,y\otimes g]:=[xy]\otimes fg+(x,y)\overline{fdg},\quad [x\otimes f,\omega]=0
$$
for $x,y\in\mathfrak g$, $f,g\in R$,  $\omega\in \Omega_R^1/dR$ and $(-,-)$ denotes the Killing form on $\mathfrak g$.  Here $\overline{a}$ denotes the image of $a\in\Omega^1_R$ in the quotient $\Omega^1_R/dR$.      A somewhat vague but natural question comes to mind
is whether there exists free field or Wakimoto type realizations of these algebras.  It is well known from the work of M. Wakimoto and B. Feigin and E. Frenkel what the answer is when $R$ is the ring of Laurent polynomials in one variable (see \cite{W} and \cite{MR92f:17026}).  We find such a realization in the setting where $\mathfrak g=\mathfrak{sl}(2,\mathbb C)$ and $R=\mathbb C[t,t^{-1},u|u^2=t^2+4t]$ is the three point algebra.

   In Kazhdan and Luszig's explicit study of the tensor structure of modules for affine Lie algebras (see \cite{MR1186962} and \cite{MR1104840}) the ring of functions regular everywhere except at a finite number of points appears naturally.   This algebra M. Bremner gave the name {\it $n$-point algebra}.  In particular in the monograph  \cite[Ch. 12]{MR1849359} algebras of the form $\oplus _{i=1}^n\mathfrak g((t-x_i))\oplus\mathbb Cc$ appear in the description of the conformal blocks.  These contain the $n$-point algebras $ \mathfrak g\otimes \mathbb C[(t-x_1)^{-1},\dots, (t-x_N)^{-1}]\oplus\mathbb Cc$ modulo part of the center $\Omega_R/dR$.   M. Bremner explicitly described the universal central extension of such an algebra in \cite{MR1261553}.  
   
   Consider now the Riemann sphere $\mathbb  C\cup\{\infty\}$ with coordinate function $s$ and fix three distinct points $a_1,a_2,a_3$ on this Riemann sphere.    Let $R$ denote the ring of rational functions with poles only in the set $\{a_1,a_2,a_3\}$.  It is known that the automorphism group $PGL_2(\mathbb C)$ of $\mathbb C(s)$ is simply 3-transitive and $R$ is a subring of $\mathbb C(s)$, so that $R$ is isomorphic to the ring of rational functions with poles at $\{\infty,0,1,a\}$.  Motivated by this isomorphism one sets $a=a_4$ and here the {\it $4$-point ring} is $R=R_a=\mathbb C[s,s^{-1},(s-1)^{-1},(s-a)^{-1}]$ where $a\in\mathbb C\backslash\{0,1\}$.    Let $S:=S_b=\mathbb C[t,t^{-1},u]$ where $u^2=t^2-2bt+1$ with $b$ a complex number not equal to $\pm 1$.  Then M. Bremner has shown us that $R_a\cong S_b$. 
As the later, being $\mathbb Z_2$-graded, is a cousin to super Lie algebras, and  is thus more immediately amendable to the theatrics of conformal field theory.  Moreover Bremner has given an explicit  description of the universal central extension of $\mathfrak g\otimes R$, in terms of ultraspherical (Gegenbauer) polynomials where $R$ is the four point algebra (see \cite{MR1249871}).      In \cite{MR2373448} the first author gave a realization for the four point algebra where the center acts nontrivially. 

In his study of the elliptic affine Lie algebras, $\mathfrak{sl}(2, R)  \oplus\left( \Omega_R/dR\right)$ where $R=\mathbb C[x,x^{-1},y\,|\,y^2=4x^3-g_2x-g_3]$, M. Bremner has also explicitly described the universal central extension of this algebra in terms of Pollaczek polynomials (see  \cite{MR1303073}).   Essentially the same algebras appear in recent work of A. Fialowski and M. Schlichenmaier \cite{FailS} and \cite{MR2183958}.  Together with Andr\'e Bueno and Vyacheslav Futorny, the first author described free field type realizations of the elliptic Lie algebra where $R=\mathbb C[t,t^{-1},u\,|, u^2=t^3-2bt^2-t]$, $b\neq \pm 1$ (see \cite{MR2541818}).

 Below we look at the three point algebra case where $R$ denotes the ring of rational functions with poles only in the set $\{a_1,a_2,a_3\}$. This algebra is isomorphic to $\mathbb C[s,s^{-1},(s-1)^{-1}]$. M. Schlichenmaier has a slightly different description of the three point algebra as $\mathbb C[(z^2-a^2)^k,z(z^2-a^2)^k\,|\, k\in\mathbb Z]$ where $a\neq 0$ (see \cite{MR2058804}). 
We show that $R\cong \mathbb C[t,t^{-1},u\,|\,u^2=t^2+4t]$ and thus looks more like $S_b$ above.
Our main result \thmref{mainthm} provides a natural free field realization in terms of a $\beta$-$\gamma$-system and the oscillator algebra of the three point affine Lie algebra when $\mathfrak g=\mathfrak{sl}(2,\mathbb C)$.   Just as in the case of intermediate Wakimoto modules defined in
\cite{ MR2271362}, there are two different realizations depending on two different normal orderings.  Besides M. Bermner's article mentioned above, other work on the universal central extension of $3$-point algebras can be found in \cite{MR2286073}. Previous related work on highest weight modules of $\mathfrak{sl}(2,R)$ can be found in H. P. Jakobsen and V. Kac \cite{JK}.

The three point algebra is perhaps the simplest non-trivial example of a Krichever-Novikov algebra beyond an affine Kac-Moody algebra 
 (see \cite{MR902293}, \cite{MR925072}, \cite{MR998426}).  A fair amount of interesting and 
fundamental work has be done by Krichever, Novikov, Schlichenmaier, and Sheinman on the representation theory of the Krichever-Novikov algebras. 
 In particular Wess-Zumino-Witten-Novikov theory and analogues of the Knizhnik-Zamolodchikov equations are developed for  these algebras 
(see the survey article \cite{MR2152962}, and for example \cite{MR1706819}, \cite{MR1706819},\cite{MR2072650},\cite{MR2058804},\cite{MR1989644}, and \cite{MR1666274}). 

The initial motivation for the use of Wakimoto's realization was to prove a conjecture of V. Kac and D. Kazhdan on the character of certain irreducible representations of affine Kac-Moody algebras at the critical level (see \cite{W} and \cite{MR2146349}). Another motivation for constructing free field realizations is that they are used to provide integral solutions to the KZ-equations (see for example \cite{MR1077959} and \cite{MR1629472} and their references).  A third is that they are used to help in determining the center of a certain completion of the enveloping algebra of an affine Lie algebra at the critical level which is an important ingredient in the geometric Langland's correspondence \cite{MR2332156}.  Yet a fourth is that free field realizations of an affine Lie algebra appear naturally in the context of the generalized AKNS hierarchies \cite{MR1729358}.

The authors would like to thank Murray Bremner for bringing this problem to their attention and they would also like to thank the College of Charleston for their summer Research and Development Grant which supported this work.    We would also like to thank Renato Martins for helpful commentary. 

\dedicatory{This paper is dedicated to Robert Wilson.}

\section{The $3$-point ring.}
The three point algebra has at least four incarnations. 
\subsection{Three point algebras}  Fix $0\neq a\in\mathbb C$.
Let \begin{align*}
\mathcal S&:=\mathbb C[s,s^{-1},(s-1)^{-1} ] ,\\
 \mathcal R&:=\mathbb C[t,t^{-1},u\,|\, u^2=t^2+4t], \\
 \mathcal A&:=\mathcal A_a=\mathbb C[(z^2-a^2)^k,z(z^2-a^2)^k\,|\,k\in\mathbb Z].
 \end{align*}
 Note that Bremner introduced the ring $\mathcal S$ and Schichenmaier introduced $\mathcal A$ (see \cite{MR2058804}).  Variants of $\mathcal R$ were introduced by Bremner for elliptic and $3$-point algebras.
\begin{prop} 
\begin{enumerate}
\item The rings $\mathcal R$ and $\mathcal S$ are isomorphic whereby $t\mapsto s^{-1}(s-1)^2$, and $u\mapsto s-s^{-1}$.
\item The rings $\mathcal R$ and $\mathcal A$ are isomorphic.
\end{enumerate}
\end{prop}
\begin{proof} 
(1). Let $\bar f: \mathbb C[t,u] \rightarrow \mathcal S$ be the ring homomorphism defined $\bar f(t)=s^{-1}(s-1)^2=s-2+s^{-1}$, $\bar f(u)= s-s^{-1}$.

 We first check that 
$$
\bar f(u^2-(t^2+4t))=(s-s^{-1})^2-(s-2+s^{-1})^2-4(s-2+s^{-1})=0
$$ 
and $\bar f(t)=s^{-1}(s-1)^2$ is invertible in $\mathcal S$.  Hence the map $\bar f$ descends to a well defined ring homomorphism $f:\mathcal R\to \mathcal S$.  To show that it is onto we essentially solve for $s$ and $s^{-1}$ in terms of $t$ and $u$.  The inverse ring homomorphism of $f$ is $\phi:\mathcal S\to \mathcal R$ given by 
$$
\phi(s)=\frac{t+2+u}{2},\quad \phi(s^{-1})=\frac{t+2-u}{2}.
$$
In particular $\displaystyle{\phi((s-1)^{-1})=\frac{t^{-1}u-1}{2}}$.

For part (2) observe $\mathcal A=\mathbb C[z,(z-a)^{-1},(z+a)^{-1}]$ which after setting $z=2as-a$ we get $\mathcal A=\mathbb C[s,s^{-1},(s-1)^{-1}]$.    Thus an isomorphism between $\mathcal A$ and $\mathcal R$ is implemented by the assignment $z\mapsto a(t+u)$, $(z+a)^{-1}\mapsto\displaystyle{\frac{t+2-u}{4a}}$ and $(z-a)^{-1}\mapsto   \displaystyle{ \frac{t^{-1}u-1}{4a}}$.  \end{proof}

\subsection{The Universal Central Extension of the Current Algebra $\mathfrak g\otimes \mathcal A$.}
 Let $R$ be a commutative algebra defined over $\mathbb C$.
Consider the left $R$-module  $F=R\otimes R$ with left action given by $f( g\otimes h ) = f g\otimes h$ for $f,g,h\in R$ and let $K$  be the submodule generated by the elements $1\otimes fg  -f \otimes g -g\otimes f$.   Then $\Omega_R^1=F/K$ is the module of {\it K\"ahler differentials}.  The element $f\otimes g+K$ is traditionally denoted by $fdg$.  The canonical map $d:R\to \Omega_R^1$ is given by $df  = 1\otimes f  + K$.  The {\it exact differentials} are the elements of the subspace $dR$.  The coset  of $fdg$  modulo $dR$ is denoted by $\overline{fdg}$.  As C. Kassel showed the universal central extension of the current algebra $\mathfrak g\otimes R$ where $\mathfrak g$ is a simple finite dimensional Lie algebra defined over $\mathbb C$, is the vector space $\hat{\mathfrak g}=(\mathfrak g\otimes R)\oplus \Omega_R^1/dR$ with Lie bracket given by
$$
[x\otimes f,Y\otimes g]=[xy]\otimes fg+(x,y)\overline{fdg},  [x\otimes f,\omega]=0,  [\omega,\omega']=0,
$$
  where $x,y\in\mathfrak g$, and $\omega,\omega'\in \Omega_R^1/dR$ and $(x,y)$  denotes the Killing  form  on $\mathfrak g$.  
  
There are at least four incarnations of the three point algebras, three of which are defined as $\mathfrak g\otimes R\oplus \Omega_R/dR$ where $R=\mathcal S, \mathcal R,\mathcal A$ given above.  The forth incarnation appears in the work of G. Benkart and P. Terwilliger given in terms of the tetrahedron algebra (see \cite{MR2286073}) . 
We will only work with $R=\mathcal R$.  

\begin{prop}[\cite{MR1261553}, see also \cite{MR1249871}]\label{uce}  Let $\mathcal R$ be as above.  The set 
$$
\{\omega_0:=\overline{t^{-1} dt},\enspace \omega_1:=\overline{t^{-1}u\,dt}\}
$$
 is a basis of $\Omega_\mathcal R^1/d\mathcal R$.
\end{prop}
\begin{proof}  The proof follows almost exactly along the lines of \cite{MR1249871} and \cite{MR1261553} and would be omitted if not for the fact that we need some of the formulae that appear in the proof.  As $\mathcal R$ has basis $\{t^k,t^lu\,|\,k,l\in\mathbb Z\}$, $\Omega_{\mathcal R}$ has a spanning set given by the image in the quotient $F/K$ of the tensor product of these basis elements.  We know $\frac{1}{2}u\,d(u^2)=u^2\,du$.    Next we observe that since $u^2=t^2+4t$ one has in $\Omega_{\mathcal R}$,
\begin{align*}
u(t+2)\,dt=\frac{1}{2}u(2t+4)\,dt=\frac{1}{2}u\,d(u^2)=u^2\,du=(t^2+4t)\,du
\end{align*}
and after multiplying this on the left by $t^k$ we get
\begin{align}\label{recursion}
(t^{k+1}+2t^k)u\,dt-(t^{k+2}+4t^{k+1})\,du=0
\end{align}
in $\Omega_{\mathcal R}$.
Now
\begin{align}
d(t^k)&=kt^{k-1}\,dt, \notag \\
d(t^ku)&=t^k\enspace du+ kt^{k-1}u\,dt, \label{eqn1}
\end{align}
so that in the quotient $\Omega_{\mathcal R}/d\mathcal R$ we have by \eqnref{recursion} followed by \eqnref{eqn1}
\begin{align} 
0&\equiv (t^{k+1}+2t^k)u\,dt-(t^{k+2}+4t^{k+1})\,du\notag \\
&\equiv (t^{k+1}+2t^k)u\,dt-(-(k+2)t^{k+1} -4(k+1)t^{k} )u\,dt \notag\\
&\equiv  \left((k+3)t^{k+1}+(4k+6)t^k\right)u\,dt. \notag
\end{align}
Then 
\begin{equation}
t^k u\,dt\equiv -\frac{(k+3)}{(4k+6)}t^{k+1}u\, dt \mod d\mathcal R\label{eqn2}
\end{equation}
so that 
$$
t^{-k} u\,dt\equiv t^{-3} u\,dt\equiv 0\mod d\mathcal R,\quad k\geq 3,
$$
and 
\begin{equation}
t^{k+1}u\,dt\equiv - \frac{(4k+6)}{(k+3)}t^k u\,dt\mod   d\mathcal R,  \label{eqn2.5}
\end{equation}
so that $t^{k}u\,dt$ can be written in terms of $t^{-1}u\,dt$ for $k\geq -2$ modulo $d\mathcal R$. 
Thus $\Omega_{\mathcal R}$ is spanned as a left $\mathcal R$-module by $dt $ and $du$, furthermore 
\begin{align}
t^{k-1}\,dt&\equiv \frac{1}{k}d(t^k)\equiv 0\mod d\mathcal R, \text{ for }k\neq 0\label{eqn3}
\end{align}
By equations \eqnref{eqn1}, \eqnref{eqn2}, and \eqnref{eqn3} we have  $\Omega_{\mathcal R}/d\mathcal R$ is spanned by  $\{\overline{t^{-1}\,dt},\overline{t^{-1}u\,dt}\}$.
We know by the Riemann-Roch Theorem that the dimension of this space of K\"ahler differentials modulo exact forms on the sphere with three punctures has dimension 2 (see \cite{MR1261553}).  This completes the proof of the proposition.
\end{proof}
Note by \eqnref{eqn2} and \eqnref{eqn2.5}
\begin{equation}
\overline{t^k u\,dt}= \frac{(-1)^{k+1}2^{k+1}(2k+1)!!}{(k+2)!}\overline{t^{-1}u\, dt}
\end{equation}
where this gives us $0$ if $k\leq -3$. 
\begin{cor} In $\Omega_R^1/dR$, one has
\begin{align}
\label{recursionreln}
\overline{t^k\,dt^l}&=-k\delta_{l,-k}\omega_0 , \\
 \overline{t^ku\,d(t^lu)}&=\left((l+1)\delta_{k+l,-2}+(4l +2)\delta_{k+l,-1}\right)\omega_0\\  
\overline{t^k\,d(t^lu)}&= \mu_{k,l}\omega_1
\end{align}
where $$
\mu_{k,l}:= -k\frac{(-1)^{k+l}2^{k+l}(2(k+l)-1)!!}{(k+l+1)!}.
$$

\end{cor}
\begin{proof} 
Using \eqnref{eqn1} above we obtain 
\begin{align*}
t^k\,d(t^lu)&\equiv t^k\enspace (lt^{l-1}u\,dt+t^l\,du) \\
&\equiv  lt^{l+k-1}u\,dt+t^{l+k}\,du \\
&\equiv  lt^{l+k-1}u\,dt-(l+k)t^{l+k-1}u\,dt \\ 
&\equiv  -kt^{l+k-1}u\,dt   \\
&\equiv -k\frac{(-1)^{k+l}2^{k+l}(2(k+l)-1)!!}{(k+l+1)!} t^{-1}u\, dt\mod dR  \\
\end{align*}
in $\Omega_{\mathcal R}/d\mathcal R$.  

Next we observe 
$\displaystyle{
u\,du=\frac{1}{2}\,d(u^2)=\frac{1}{2}\,d(t^2+4t)=(t+2)\,dt}$
so that in $\Omega_{\mathcal R}$, 
\begin{equation}
t^ku\,du=(t^{k+1}+2t^k)\,dt\label{udu}.
\end{equation}
By \eqnref{udu} and \eqnref{eqn3}
\begin{align*}
t^ku\,d(t^lu)&= t^ku(lt^{l-1}u\,dt+t^l\,du) \hskip 20pt \text{ in }\Omega_{\mathcal R} \\
&= (lt^{l+k-1}u^2\,dt+t^{l+k}u\,du) \\ 
&= (lt^{l+k-1}(t^2+4t)\,dt+(t^{l+k+1}+2t^{l+k})\,dt) \\ 
&=  l (t^{k+l+1}+4t^{k+l})\,dt+(t^{l+k+1}+2t^{l+k})\,dt) \\ 
&= (l+1) t^{k+l+1}\,dt+(4l+2)t^{k+l}\,dt  \\ 
&\equiv\left((l+1)\delta_{k+l,-2}+(4l +2)\delta_{k+l,-1}\right)t^{-1}\,dt \mod
\mathcal R.
\end{align*}
This completes the proof of the corollary.
\end{proof}

\begin{thm}\label{3ptthm}
The universal central extension of the algebra $\mathfrak{sl}(2,\mathbb C)\otimes \mathcal R$ is isomorphic to the Lie algebra with generators $e_n$, $e_n^1$, $f_n$, $f_n^1$, $h_n$, $h_n^1$, $n\in\mathbb Z$, $\omega_0$, $\omega_1$ and relations given by
\begin{align}
[x_m,x_n]&:=[x_m,x_n^1]=[x_m^1,x_n^1]=0,\quad \text{ for }x=e,f\label{xs}\\
[h_m,h_n]&:=-2m\delta_{m,-n}\omega_0=(n-m)\delta_{m,-n}\omega_0 , \\
[h^1_m,h^1_n]&:=2\left((n+1)\delta_{m+n,-2}+(4n +2)\delta_{m+n,-1}\right)\omega_0 \\
&=(n-m)\left(\delta_{m+n,-2}+4\delta_{m+n,-1}\right)\omega_0,\notag \\
[h_m,h_n^1]&:=-2\mu_{m,n}\omega_1 ,\\
 [\omega_i,x_m]&=[\omega_i,\omega_j]=0,\quad \text{ for }x=e,f,h,\quad i,j\in\{0,1\} \label{omega1} \\ 
[e_m,f_n]&=h_{m+n}-m\delta_{m,-n}\omega_0, \label{ef}\\
[e_m,f_n^1]&=h^1_{m+n}-\mu_{m,n}\omega_1=:[e_m^1,f_n], \\
[e_m^1,f_n^1]&:=h_{m+n+2}+4h_{m+n+1}+\left((n+1)\delta_{m+n,-2}+(4n +2)\delta_{m+n,-1}\right)\omega_0  \\
 &=h_{m+n+2}+4h_{m+n+1}+\frac{1}{2}(n-m)\left(\delta_{m+n,-2}+4\delta_{m+n,-1}\right)\omega_0,\notag  \\
[h_m,e_n]&:=2e_{m+n}, \\
[h_m,e^1_n]&:=2e^1_{m+n} = :[h_m^1,e_m],  \\
[h_m^1,e_n^1]&:=2e_{m+n+2} +8e_{m+n+1},\label{he}\\
[h_m,f_n]&:=-2f_{m+n}, \\
[h_m,f^1_n]&:=-2f^1_{m+n} =:[h_m^1,f_m], \\
[h_m^1,f_n^1]&:=-2f_{m+n+2} -8f_{m+n+1} ,\label{last}
\end{align}
for all $m,n\in\mathbb Z$.
\end{thm}

\begin{proof}  Let $\mathfrak f$ denote the free Lie algebra with the generators  $e_n$, $e_n^1$, $f_n$, $f_n^1$, $h_n$, $h_n^1$, $n\in\mathbb Z$, $\omega_0$, $\omega_1$ and relations given above \eqnref{xs}-\eqnref{last}.  The map
$\phi:\mathfrak f\to (\mathfrak{sl}(2,\mathbb C)\otimes \mathcal R)\oplus ( \Omega_{\mathcal R}/d\mathcal R)$ given by 
\begin{align*}
\phi(e_n):&=e\otimes t^n,\quad \phi(e_n^1)=e\otimes ut^n, \\
\phi(f_n):&=f\otimes t^n,\quad \phi(f_n^1)=f\otimes ut^n, \\
\phi(h_n):&=h\otimes t^n,\quad \phi(h_n^1)=h\otimes ut^n, \\
\phi(\omega_0):&=\overline{t^{-1}\,dt},\quad \phi(\omega_1)=\overline{t^{-1}u\,dt}, \\
\end{align*}
for $n\in\mathbb Z$ is a surjective Lie algebra homomorphism.  

 Consider the subalgebras $S_+=\langle e_n,e_n^1\,|\,n\in\mathbb Z\rangle$, $S_0=\langle h_n,h_n^1,\omega_0,\omega_1\,|\,n\in\mathbb Z\rangle$
and $S_-=\langle f_n,f_n^1\,|\,n\in\mathbb Z\rangle$ and set $S=S_-+S_0+S_+$.   By \eqnref{xs} -\eqnref{omega1} we have
$$
S_+=\sum_{n\in\mathbb Z}\mathbb Ce_n+\sum_{n\in\mathbb Z}\mathbb Ce_n^1,\quad S_-=\sum_{n\in\mathbb Z}\mathbb C f_n+\sum_{n\in\mathbb Z}\mathbb Cf_n^1,\quad S_0=\sum_{n\in\mathbb Z}\mathbb C h_n+\sum_{n\in\mathbb Z}\mathbb Ch_n^1+\mathbb C\omega_0+\mathbb C\omega_1
$$
By \eqnref{ef}-\eqnref{he} we see that 
$$
[e_n,S_+]=[e_n^1,S_+]=0,\quad [h_n,S_+]\subseteq S_+,\quad [h_n^1,S_+]\subseteq S_+,,\quad [f_n,S_+]\subseteq S_0,\quad [f_n^1,S_+]\subseteq S_0.
$$
and similarly 
$[x_n,S_-]=[x_n^1,S_-] \subseteq  S$, $[x_n,S_0]=[x_n^1,S_0] \subseteq  S$ for $x=e,f,h$. 
To sum it up we observe that $[x_n,S]\subseteq S$ and $[x_n^1,S]\subseteq S$ for $n\in\mathbb Z$, $x=h,e,f$.
Thus $[S,S]\subset S$.
Now $S$ contains the generators of $\mathfrak f$ and is a subalgebra.  Hence $S=\mathfrak f$.  Now it is clear that $\phi$ is a Lie algebra isomorphism.
\end{proof}

\section{A triangular decomposition of the $3$-point loop algebras $\mathfrak g\otimes R$}

From now on we identify $R_a$ with $\mathcal S$ and set $R=\mathcal S$ which has a basis $t^i,t^iu$, $i\in\mathbb Z$.  Let $p:R\to R$ be the automorphism given by $p(t)=t$ and $p(u)=-u$.   Then one can decompose $R=R^0\oplus R^1$ where $R^0=\mathbb C[t^{\pm1}]=\{r\in R\,|\, p(r)=r\}$ and
$R^1=\mathbb C[t^{\pm1}]u=\{r\in R\,|\, p(r)=-r\}$ are the eigenspaces of $p$.
From now on $\mathfrak g$ will denote a simple Lie algebra over $\mathbb C$ with triangular decomposition $\mathfrak g=\mathfrak n_-\oplus \mathfrak h\oplus\mathfrak n_+$ and then the $3$-{\it point loop algebra} $L({\mathfrak g}):=\mathfrak g\otimes R$ has a corresponding $\mathbb Z/2\mathbb Z$-grading:  $L({\mathfrak g})^i:=\mathfrak g\otimes R^i$ for $i=0,1$.  However the degree of $t$ does not render $L({\mathfrak g})$ a $\mathbb Z$-graded Lie algebra.  This leads us to the following notion.

Suppose $I$ is an additive subgroup of the rational numbers $\mathbb P$ and $\mathcal A$ is a $\mathbb C$-algebra such that $\mathcal A=\oplus_{i\in I}\mathcal A_i$  and there exists a fixed $l\in\mathbb N$, with
$$
\mathcal A_i\mathcal A_j\subset \oplus_{|k-(i+j)|\leq l}\mathcal A_k
$$
for all $i,j\in\mathbb Z$.  Then $\mathcal A$ is said to be an {\it $l$-quasi-graded algebra}.  For $0\neq x\in \mathcal A_i$ one says that $x$ is {\it homogeneous of degree $i$} and one writes $\deg x=i$.

  For example $R$ has the structure of a $1$-quasi-graded algebra where $I=\frac{1}{2}\mathbb Z$ and $\deg t^i=i$, $\deg t^iu=i+\frac{1}{2}$.

A {\it weak triangular decomposition} of a Lie algebra $\mathfrak l$ is a triple $(\mathfrak H,\mathfrak l_+,\sigma)$ satisfying
\begin{enumerate}
\item $\mathfrak H$ and $\mathfrak l_+$ are subalgebras of $\mathfrak l$,
\item $\mathfrak H$ is abelian and $[\mathfrak H,\mathfrak l_+]\subset \mathfrak l_+$,
\item $\sigma$ is an anti-automorphism of $\mathfrak l$ of order $2$ which is the identity on $\mathfrak h$ and 
\item $\mathfrak l=\mathfrak l_+\oplus \mathfrak H\oplus\sigma( \mathfrak l_+)$.
\end{enumerate}
We will let $\sigma(\mathfrak l_+)$ be denoted by $\mathfrak l_-$.

\begin{thm}[\cite{MR1249871}, Theorem 2.1]  The $3$-point loop algebra $L({\mathfrak g})$ is $1$-quasi-graded Lie algebra where $\deg (x\otimes f)=\deg f$ for $f$ homogeneous.  Set $R_+=\mathbb C(1+u)\oplus \mathbb C[t,u]t$ and $R_-=p(R_+)$.  Then $L({\mathfrak g})$ has a weak triangular decomposition given by
$$
L({\mathfrak g})_\pm=\mathfrak g\otimes R_\pm,\quad \mathcal H:=\mathfrak h\otimes \mathbb C.
$$
\end{thm}

\subsection{Formal Distributions}
We need some more notation that will simplify some of the arguments
later.
This notation follows roughly \cite{MR99f:17033} and \cite
{MR2000k:17036}:  The {\it formal delta function}
$\delta(z/w)$ is the formal distribution
$$
\delta(z/w)=z^{-1}\sum_{n\in\mathbb Z}z^{-n}w^{n}=w^{-1}\sum_{n\in\mathbb Z}z^{n}w^{-n}.
$$
For any sequence of elements $\{a_{m}\}_{m\in
\mathbb Z}$ in the ring $\End (V)$, $V$ a vector space,  the
formal distribution
\begin{align*}
a(z):&=\sum_{m\in\mathbb Z}a_{m}z^{-m-1}
\end{align*}
is called a {\it field}, if for any $v\in V$, $a_{m}v=0$ for $m\gg0$.
If $a(z)$ is a field, then we set
\begin{align}\label{usualnormalordering}
    a(z)_-:&=\sum_{m\geq 0}a_{m}z^{-m-1},\quad\text{and}\quad
   a(z)_+:=\sum_{m<0}a_{m}z^{-m-1}.
\end{align}
 The {\it normal ordered product} of two distributions
$a(z)$ and
$b(w)$ (and their coefficients) is
defined by
\begin{equation}\label{normalorder}
\sum_{m\in\mathbb Z}\sum_{n\in\mathbb
Z}:a_mb_n:z^{-m-1}w^{-n-1}=:a(z)b(w):=a(z)_+b(w)+b(w)a(z)_-.
\end{equation}

Now we should point out that while $:a^1(z_1)\cdots a^m(z_m):$
is always defined as a formal series, we will only define $:a(z)
b(z)::=\lim_{w\to z}:a(z)b(w):$ 
for certain pairs
$(a(z),b(w))$.  

Then one defines recursively
\[
:a^1(z_1)\cdots a^k(z_k):=:a^1(z_1)\left(:a^2(z_2)\left(:\cdots
:a^{k-1}(z_{k-1}) a^k(z_k):\right)\cdots
:\right):,
\]
while normal ordered product
\[
:a^1(z)\cdots
a^k(z):=\lim_{z_1,z_2,\cdots, z_k\to
z} :a^1(z_1)\left(:a^2(z_2)\left(:\cdots :a^{k-1}(z_{k-1})
a^k(z_k):\right)\cdots
\right):
\]
will only be defined for certain $k$-tuples $(a^1,\dots,a^k)$.

Let 
\begin{equation}\label{contraction}
\lfloor
ab\rfloor=a(z)b(w)-:a(z)b(w):= [a(z)_-,b(w)],
\end{equation}
(half of
$[a(z),b(w)]$) denote the {\it contraction} of any two formal distributions 
$a(z)$ and $b(w)$. 
\begin{thm}[Wick's Theorem, \cite{MR85g:81096}, \cite{MR99m:81001} or  
\cite{MR99f:17033} ]  Let  $a^i(z)$ and $b^j(z)$ be formal
distributions with coefficients in the associative algebra 
 $\End(\mathbb C[\mathbf x]\otimes \mathbb C[\mathbf y])$, 
 satisfying
\begin{enumerate}
\item $[ \lfloor a^i(z)b^j(w)\rfloor ,c^k(x)_\pm]=[ \lfloor
a^ib^j\rfloor ,c^k(x)_\pm]=0$, for all $i,j,k$ and
$c^k(x)=a^k(z)$ or
$c^k(x)=b^k(w)$.
\item $[a^i(z)_\pm,b^j(w)_\pm]=0$ for all $i$ and $j$.
\item The products 
$$
\lfloor a^{i_1}b^{j_1}\rfloor\cdots
\lfloor a^{i_s}b^{i_s}\rfloor:a^1(z)\cdots a^M(z)b^1(w)\cdots
b^N(w):_{(i_1,\dots,i_s;j_1,\dots,j_s)}
$$
have coefficients in
$\End(\mathbb C[\mathbf x]\otimes \mathbb C[\mathbf y])$ for all subsets
$\{i_1,\dots,i_s\}\subset \{1,\dots, M\}$, $\{j_1,\dots,j_s\}\subset
\{1,\cdots N\}$. Here the subscript
${(i_1,\dots,i_s;j_1,\dots,j_s)}$ means that those factors $a^i(z)$,
$b^j(w)$ with indices
$i\in\{i_1,\dots,i_s\}$, $j\in\{j_1,\dots,j_s\}$ are to be omitted from
the product
$:a^1\cdots a^Mb^1\cdots b^N:$ and when $s=0$ we do not omit
any factors.
\end{enumerate}
Then
\begin{align*}
:&a^1(z)\cdots a^M(z)::b^1(w)\cdots
b^N(w):= \\
  &\sum_{s=0}^{\min(M,N)}\sum_{i_1<\cdots<i_s,\,
j_1\neq \cdots \neq j_s}\lfloor a^{i_1}b^{j_1}\rfloor\cdots
\lfloor a^{i_s}b^{j_s}\rfloor
:a^1(z)\cdots a^M(z)b^1(w)\cdots
b^N(w):_{(i_1,\dots,i_s;j_1,\dots,j_s)}.
\end{align*}
\end{thm}

For $m=i-\frac{1}{2}$, $i\in\mathbb Z+\frac{1}{2}$ and $x\in\mathfrak g$, define $x_{m+\frac{1}{2}}=x\otimes t^{i-\frac{1}{2}}u=x^1_m$ and $x_m:=x\otimes t^m$.  Motivated by conformal field theory we set
\begin{align*}
x^1(z)
&:=\sum_{m\in\mathbb Z}x_{m+\frac{1}{2}}z^{-m-1},\quad x(z):=\sum_{m\in\mathbb Z}x_{m}z^{-m-1}.
\end{align*}
Then the relations in \thmref{3ptthm} can be rewritten as
\begin{align}
[x(z),y(w)]
&= [xy](w)\delta(z/w)-(x,y)\omega_0\partial_w\delta(z/w), \label{r1} \\ 
[x^1(z),y^1(w)]
&= P\left([x,y](w)\delta(z/w) -(x,y)\omega_0\partial_w\delta(z/w)\right)-\frac{1}{2}(x,y)(\partial P)\omega_0  \delta(z/w),  \label{r2} \\ 
[x(z),y^1(w)]
&=[x,y]^1(w)\delta(z/w) -\frac{1}{2}(x,y)\omega_1\sqrt{1+(4/w)}w\partial_w\delta(z/w) =
[x^1(z),y(w)], \label{r3}
 \end{align}
where $x,y\in\{e,f,h\}$.  (In the last commutator one can think of the formal series $\sqrt{1+(4/w)}w$ as being $\sqrt{w^2+4w}=\sqrt{P(w)}$.)
This is because
 \begin{align*}
 \sum_{k,l}\mu_{k,l}z^{-k-1}w^{-l-1} &=-\sum_{k,l}  k\frac{(-1)^{k+l}2^{k+l}(2(k+l)-1)!!}{(k+l+1)!}z^{-k-1}w^{-l-1}  \\
 &=-\sum_{k+l+1\geq 0}\sum_{l\in\mathbb Z}  k\frac{(-1)^{k+l}2^{k+l}(2(k+l)-1)!!}{(k+l+1)!}w^{-k-l-1}z^{-k-1}w^k   \\
 &=-\sum_{n\geq 0}\frac{(-1)^{n-1}2^{n-1}(2n-3)!!}{n!}w^{-n}\sum_{k\in\mathbb Z}   kz^{-k-1}w^k   \\
 &=-\sum_{n\geq 0}\frac{(-1)^{n-1}2^{2n-1}(2n-3)!!}{2^nn!}w^{-n}w\partial_w\delta(z/w)   \\
 &=-\frac{1}{2}\sum_{n\geq 0}\frac{(-1)^{n-1}(2n-3)!!}{2^nn!}(w/4)^{-n}w\partial_w\delta(z/w)   \\
 &=-\frac{1}{2}\sqrt{1+(4/w)}w\partial_w\delta(z/w).
  \end{align*}

 \section{ Oscillator algebras}
 \subsection{The $\beta-\gamma$ system}   The following construction in the physics literature is often called the $\beta-\gamma$ system which corresponds to our $a$ and $a^*$ below.
 Let $\hat{\mathfrak a}$ be the infinite dimensional oscillator algebra with generators $a_n,a_n^*,a^1_n,a^{1*}_n,\,n\in\mathbb Z$ together with $\mathbf 1$ satisfying the relations 
\begin{gather*}
[a_n,a_m]=[a_m,a_n^1]=[a_m,a_n^{1*}]=[a^*_n,a^*_m]=[a^*_n,a^1_m]=[a^*_n,a^{1*}_{m}]=0,\\
[a_n^{1},a_m^{1}]=[a_n^{1*},a_m^{1*}]=0=[\mathfrak a,\mathbf 1], \\
[a_n,a_m^*]=\delta_{m+n,0}\mathbf 1=[a^1_n,a_m^{1*}].
\end{gather*}
For  $c=a,a^1$ and respectively $X=x,x^1$ with $r=0$ or $r=1$, we define $\mathbb C[\mathbf x]:= \mathbb C[x_n,x_n^1\,|\,n\in\mathbb Z$ and $\rho:\hat{\mathfrak a}\to \mathfrak{gl}(\mathbb C[\mathbf x])$ by
\begin{align}
\rho_r( c_{m}):&=\begin{cases}
  \partial/\partial
X_{m}&\quad \text{if}\quad m\geq 0,\enspace\text{and}\enspace  r=0
\\ X_{m} &\quad \text{otherwise},
\end{cases}\label{c}
 \\
\rho_r(c_{m}^*):&=
\begin{cases}X_{-m} &\enspace \text{if}\quad m\leq
0,\enspace\text{and}\enspace r=0\\ -\partial/\partial
X_{-m}&\enspace \text{otherwise}. \end{cases}\label{c*}
\end{align}
and $\rho_r(\mathbf 1)=1$.
These two representations can be constructed using induction:
For $r=0$ the representation 
$\rho_0$ is the
$\hat{\mathfrak a}$-module generated by $1=:|0\rangle$, where
$$
a_{m}|0\rangle=a^1_{m}|0\rangle=0,\quad m\geq  0,
\quad a_{m}^*|0\rangle= a_{m}^{1*}|0\rangle=0,\quad m>0.
$$
For $r=1$ the representation 
$\rho_1$ is the
$\hat{\mathfrak a}$-module generated by $1=:|0\rangle$, where
$$
\quad a_{m}^*|0\rangle= a_{m}^{1*}|0\rangle=0,\quad m\in\mathbb Z.
$$
If we write 
$$
 \alpha(z):=\sum_{n\in\mathbb Z}a_nz^{-n-1},\quad  \alpha^*(z):=\sum_{n\in\mathbb Z}a_n^*z^{-n},
$$
and
$$
 \alpha^1(z):=\sum_{n\in\mathbb Z}a^1_nz^{-n-1},\quad  \alpha^{1*}(z):=\sum_{n\in\mathbb Z}a^{1*}_nz^{-n},
$$
then 
\begin{align*}
[\alpha(z),\alpha(w)]&=[\alpha^*(z),\alpha^*(w)]=[\alpha^{1}(z),\alpha^{1}(w)]=[\alpha^{1*}(z),\alpha^{1*}(w)]=0  \\
[\alpha(z),\alpha^*(w)]&=[\alpha^1(z),\alpha^{1*}(w)]
    =\mathbf 1\delta(z/w).
\end{align*}
Observe that $\rho_1(\alpha(z))$ and 
$\rho_1(\alpha^1(z))$ are not fields whereas $\rho_r(\alpha^*(z))$ $\rho_r(\alpha^{1*}(z))$
are always a field.   
Corresponding to these two representations there are two possible normal orderings:  For $r=0$ we use the usual normal ordering given by \eqnref{usualnormalordering} and for $r=1$ we define the {\it natural normal ordering} to be 
\begin{alignat*}{2}
\alpha(z)_+&=\alpha(z),\quad &\alpha(z)_-&=0 \\
\alpha^1(z)_+&=\alpha^1(z),\quad &\alpha^1(z)_-&=0 \\
\alpha^*(z)_+&=0,\quad &\alpha^*(z)_-&=\alpha^*(z), \\
\alpha^{1*}(z)_+&=0,\quad &\alpha^{1*}(z)_-&=\alpha^{1*}(z) ,
\end{alignat*}

This means in particular that for $r=0$ we get 
\begin{align}
\lfloor \alpha(z)\alpha^*(w)\rfloor
&=\sum_{m\geq 0} \delta_{m+n,0}z^{-m-1}w^{-n}
=\delta_-(z/w)
=\ 
\,\iota_{z,w}\left(\frac{1}{z-w}\right)\\
\lfloor \alpha^*(z)\alpha(w)\rfloor
&
=-\sum_{m\geq 1} \delta_{m+n,0}z^{-m}
w^{-n-1}
=-\delta_+(w/z)=\,\iota_{z,w}\left(\frac{1}{w-z}
\right)
\end{align}
(where $\iota_{z,w}$ Taylor series expansion in the `` region'' $|z|>|w|$), 
and for $r=1$ 
\begin{align}
\lfloor \alpha\alpha^*\rfloor
&=[\alpha(z)_-,\alpha^*(w)]=0 \\
\lfloor \alpha^*\alpha\rfloor
&=[\alpha^*(z)_-,\alpha(w)]=
-\sum_{\in\mathbb Z} \delta_{m+n,0}z^{-m}
w^{-n-1}
=- \delta(w/z),
\end{align}
where similar results hold for $\alpha^1$.
Notice that in both cases we have
$$
[\alpha(z),\alpha^*(w)]=
\lfloor \alpha(z)\alpha^*(w)\rfloor-\lfloor\alpha^*(w) \alpha(z)\rfloor=\delta(z/w).
$$

We will also need the following two
results.
\begin{thm}[Taylor's Theorem, \cite{MR99f:17033}, 2.4.3]
\label{Taylorsthm}  Let
$a(z)$ be a formal distribution.  Then in the region $|z-w|<|w|$,
\begin{equation}
a(z)=\sum_{j=0}^\infty \partial_w^{(j)}a(w)(z-w)^j.
\end{equation}
\end{thm}

\begin{thm}[\cite{MR99f:17033}, Theorem 2.3.2]\label{kac}  Set $\mathbb C[\mathbf x]=\mathbb C[x_n,x^1_n|n\in\mathbb Z]$ and $\mathbb C[\mathbf y]= C[y_m,y_m^1|m\in\mathbb N^*]$.  Let $a(z)$ and $b(z)$ 
be formal distributions with coefficients in the associative algebra 
 $\End()$ where we are using the usual normal ordering.   The
following are equivalent
\begin{enumerate}[(i)]
\item
$\displaystyle{[a(z),b(w)]=\sum_{j=0}^{N-1}\partial_w^{(j)}
\delta(z-w)c^j(w)}$, where $c^j(w)\in \End(\mathbb C[\mathbf x]\otimes \mathbb 
C[\mathbf y])[\![w,w^{-1}]\!]$.
\item
$\displaystyle{\lfloor
ab\rfloor=\sum_{j=0}^{N-1}\iota_{z,w}\left(\frac{1}{(z-w)^{j+1}}
\right)
c^j(w)}$.
\end{enumerate}\label{Kacsthm}
\end{thm}

In other words the singular part of the {\it operator product
expansion}
$$
\lfloor
ab\rfloor=\sum_{j=0}^{N-1}\iota_{z,w}\left(\frac{1}{(z-w)^{j+1}}
\right)c^j(w)
$$
completely determines the bracket of mutually local formal
distributions $a(z)$ and $b(w)$.   One writes
$$
a(z)b(w)\sim \sum_{j=0}^{N-1}\frac{c^j(w)}{(z-w)^{j+1}}.
$$

\subsection{The $3$-point Heisenberg algebra} 
The Cartan subalgebra $\mathfrak h$ tensored with $\mathcal R$ generates a subalgebra of $\hat{{\mathfrak g}}$ which is an extension of an oscillator algebra.    This extension motivates the following definition:  The Lie algebra with generators $b_{m},b_m^1$, $m\in\mathbb Z$, $\mathbf 1_0,\mathbf 1_1 $, and relations
\begin{align}
[b_{m},b_{n}]&=(n-m)\,\delta_{m+n,0}\mathbf 1_0=-2m\,\delta_{m+n,0}\mathbf 1_0\label{b1} \\
[b^1_m,b_n^1] &=(n-m)\left(\delta_{m+n,-2}+4\delta_{m+n,-1}\right)\mathbf 1_0 = 2 \left((n+1)\delta_{m+n,-2}+(4n+2)\delta_{m+n,-1}\right)\mathbf 1_0\label{b2}\\
[b_n, b^1_m] & =  2\mu_{n,m} \mathbf 1_1  = - [b^1_m,b_n] 
\label{b3} \\
[b_{m},\mathbf 1_0]&=[b_{m}^1,\mathbf 1_0]=[b_{m},\mathbf 1_1]=[b_{m}^1,\mathbf 1_1]= 0.\label{b4}
\end{align}
we will give the appellation the {\it $3$-point (affine) Heisenberg algebra} and denote it by $\hat{\mathfrak h}_3$.

If we introduce the formal distributions
\begin{equation} 
\beta(z):=\sum_{n\in\mathbb Z} b_nz^{-n-1},\quad \beta^1(z):=\sum_{n\in\mathbb Z}b_n^1z^{-n-1}=\sum_{n \in\mathbb Z}b_{n+\frac{1}{2}}z^{-n-1}.
\end{equation}
(where $b_{n+\frac{1}{2}}:=b^1_n$)
then using calculations done earlier for the $3$-point Lie algebra we can see that the relations above can be rewritten in the form
\begin{align*}\label{bosonrelations}
[\beta(z),\beta(w)]&=2\mathbf 1_0\partial_z\delta(z/w)=-2\partial_w\delta(z/w)\mathbf 1_0 \\
[\beta^1(z),\beta^1(w)]
&=-2\left((w^2+4w) \partial_w(\delta(z/w)+ (2+w) \delta(z/w)\right)\mathbf 1_0 \\
[\beta(z),\beta^1(w)]&= -\sqrt{1+(4/w)}w \partial_w\delta(z/w)\mathbf 1_1
\end{align*}

 Set
\begin{align*}
\hat{\mathfrak h}_3^\pm:&=\sum_{n\gtrless 0}\left(\mathbb Cb_n+\mathbb Cb_n^1\right),\quad
\hat{ \mathfrak h}_3^0:=  \mathbb C\mathbf 1_0\oplus \mathbb C\mathbf 1_1\oplus \mathbb Cb_0\oplus \mathbb Cb^1_0.
\end{align*}
We introduce a Borel type subalgebra
\begin{align*}
\hat{\mathfrak b}_3&= \hat{\mathfrak h}_3^+\oplus \hat{\mathfrak h}_3^0.
\end{align*}
Due to the defining relations above one can see that $\hat{\mathfrak b}_3$ is a subalgebra.

\begin{lem}\label{heisenbergprop}
Let $\mathcal V=\mathbb C\mathbf v_0\oplus \mathbb C\mathbf v_1$ be a two dimensional representation of $\hat{\mathfrak h}_3^+\mathbf v_i=0$ for $i=0,1$.   Suppose  $\lambda,\mu,\nu,\varkappa, \chi_1,\kappa_0 \in \mathbb C$  are such that
\begin{align*}
b_0\mathbf v_0&=\lambda \mathbf v_0,  &b_0\mathbf v_1&=\lambda \mathbf v_1 \\
b_0^1\mathbf v_0&=\mu \mathbf v_0+\nu \mathbf v_1,  &b_0^1\mathbf v_1&=\varkappa \mathbf v_0+\mu \mathbf v_1\\
\mathbf 1_1\mathbf v_i&=\chi _1  \mathbf v_i,\quad   &\mathbf 1_0\mathbf v_i&=\kappa _0\mathbf v_i,\quad i=0,1.
\end{align*}
Then the above defines a representation of  $\hat{\mathfrak b}_3$ if $\chi_1 = 0$.
\end{lem}
\begin{proof} Since $b_m$ acts by scalar multiplication for $m,n\geq 0$, the first defining relation \eqnref{b1} is satisfied for $m,n\geq 0$. 
The second relation \eqnref{b2} is also satisfied as the right hand side is zero if $m\geq 0,n\geq 0$.  If $n=0$, then since $b_0$ acts by a scalar, the relation \eqnref{b3} leads to no condition on $\lambda,\mu,\nu,\varkappa, \chi_1,\kappa_0 \in \mathfrak h_4^0$.    If $m\geq 0$ and $n>0$, the third relation gives the condition on $\chi_1$ as 
$$
0=b^1_mb_n\mathbf v_i-b_nb^1_m\mathbf v_i=[b^1_m,b_n] \mathbf v_i=-2\mu_{n,m}\chi_1\mathbf v_i
$$
which forces $\chi_1 = 0$.

\end{proof}
Let $B_0^1$ denote the linear transformation on $\mathcal V$ that agrees with the action of $b_0^1$.
If we define the notion of a $\widehat{\mathfrak b}_3$-submodule as is done in \cite{MR1328538}, Definition 1.2, then $\mathcal V$ above is an irreducible $\widehat{\mathfrak b}_3$-module when $\varkappa \nu\neq 0$ i.e. if $\det B_0^1\neq \mu^2$.  If one induces up from $\mathcal V$, the resulting representation for the three point affine algebra would not have a chance of being irreducible if $\mathcal V$ were not irreducible (in the sense of Sheinman) itself. 

Let $ \mathbb C[\mathbf y]:= \mathbb C[y_{-n}, y_{-m}^1 | m,n \in \mathbb N^*] $.

\begin{lem} \label{rhorep}The linear map $\rho:\hat{\mathfrak b}_3\to \text{End}(\mathbb C[\mathbf y]\otimes \mathcal V)$ defined  by 
\begin{align}
\rho(b_{n})&=y_{n} \quad \text{ for }n<0 \\
\rho(b_{n}^1)&=y_{n}^1+\delta_{n,-1}\partial_{y_{-3}^1}\chi_0-\delta_{n,-3}\partial_{y_{-1}^1}\chi_0\quad \text{ for }n<0 \\
\rho (b_n) &= -2n  \partial_{ y_{-n} }\chi_0   \quad \text{ for }n>0 \\
\rho(b^1_n)&=  2(n+2) \partial_{y^1_{-n-4}} \chi_0-4c(n+1) \partial_{y^1_{-n-2}} \chi_0 +2n \partial_{y^1_{-n}} \chi_0  \quad \text{ for }n>0\\
\rho(b_{0})&=\lambda \\
\rho(b_{0}^1)&=4\partial_{y_{-4}^1}\chi_0-2c\partial_{y_{-2}^1}\chi_0 +B_0^1.
\end{align}
is a representation of $\hat{\mathfrak b}_3$.
\end{lem}
\begin{proof}
For $m,n> 0$, it is straight forward to see $[\rho(b_n),\rho(b_m)]=[\rho(b^1_n),\rho(b^1_m)]=0$, and similarly for $m,n<0$, $[\rho(b_n),\rho(b_m)]=0$ and $[\rho(b^1_n),\rho(b^1_m)]=0$ if $n\notin\{-1,-3\}$. But
\begin{align*}
[\rho(b^1_{-1}),\rho(b^1_m)]&=[y_{-1}^1+\partial_{y_{-3}^1}\chi_0,y_{m}^1+\delta_{m,-1}\partial_{y_{-3}^1}\chi_0-\delta_{m,-3}\partial_{y_{-1}^1}\chi_0]\\
&=-\delta_{m,-3}\chi_0[y_{-1}^1,\partial_{y_{-1}^1}]\chi_0+\delta_{m,-3}[\partial_{y_{-3}^1},y_{-3}^1]\chi_0\\
&=-2\delta_{m,-3}\chi_0
\end{align*}
\begin{align*}
[\rho(b^1_{-3}),\rho(b^1_m)]&=[y_{-3}^1-\partial_{y_{-1}^1}\chi_0,y_{m}^1+\delta_{m,-1}\partial_{y_{-3}^1}\chi_0-\delta_{m,-3}\partial_{y_{-1}^1}\chi_0]\\
&=\delta_{m,-1}\chi_0[y_{-3}^1,\partial_{y_{-3}^1}]\chi_0-\delta_{m,-1}[\partial_{y_{-1}^1},y_{-1}^1]\chi_0\\
&=2\delta_{m,-1}\chi_0
\end{align*}
\begin{align*}
[\rho(b^1_{0}),\rho(b^1_m)]&=[4\partial_{y_{-4}^1}\chi_0-2c\partial_{y_{-2}^1}\chi_0,y_{m}^1+\delta_{m,-1}\partial_{y_{-3}^1}\chi_0-\partial_{m,-3}\partial_{y_{-1}^1}\chi_0]\\
&=-4\delta_{m,-4}\chi_0+2c\partial_{m,-2}\chi_0
\end{align*}

For $m>0$ and $n\leq 0$ we have 
\begin{align*}
[\rho(b_m),\rho(b_n)]&=[-m (2 \partial_{ y_{-m} }\chi_0   +  2\partial_{y_{-m}^1} \chi_1) ,y_n]  \\ 
&=-2m\delta_{m,-n}\chi_0,  \\ \\ 
[\rho(b_m),\rho(b^1_n)]&=[-m (2 \partial_{ y_{-m} }\chi_0   +  2\partial_{y_{-m}^1} \chi_1) ,y^1_n+\delta_{n,-1}\partial_{y_{-3}^1}\chi_0-\delta_{n,-3}\partial_{y_{-1}^1}\chi_0]  \\ 
&=-2m\delta_{m,-n}\chi_1 ,    \\  \\
[\rho(b^1_m),\rho(b^1_n)]&=[-2n\partial_{y_{-n}} \chi_1 +2(n+2) \partial_{y^1_{-n-4}} \chi_0-4c(n+1) \partial_{y^1_{-n-2}} \chi_0 +2n \partial_{y^1_{-n}} \chi_0\\
&\hspace{1.0cm} ,y^1_n+\delta_{n,-1}\partial_{y_{-3}^1}\chi_0-\delta_{n,-3}\partial_{y_{-1}^1}\chi_0]  \\ 
&= 2(n+2) \delta_{m+n,-4} \chi_0 -4c(n+1) \delta_{m+n,-2} \chi_0+2n\delta_{m+n,0}\chi_0,      \\  \\
[\rho(b^1_m),\rho(b_n)]&=[-2m\partial_{y_{-m}} \chi_1 +2(m+2) \partial_{y^1_{-m-4}} \chi_0-4c(m+1) \partial_{y^1_{-m-2}} \chi_0 +2m \partial_{y^1_{-m}} \chi_0 \\
&\hspace{1.0cm} ,y_n+\delta_{n,-1}\partial_{y_{-3}^1}\chi_0-\delta_{n,-3}\partial_{y_{-1}^1}\chi_0]  \\ 
&= -2m\delta_{m,-n}\chi_1  .
\end{align*}

\end{proof}
\color{black}

\section{Two realizations of the affine $3$-point algebra $\hat{{\mathfrak g}}$}
Our main result is the following 
\begin{thm} \label{mainthm}  Fix $r\in\{0,1\}$, which then fixes the corresponding normal ordering convention defined in the previous section.  Set $\hat{{\mathfrak g}} =\left(\mathfrak{sl}(2,\mathbb C)\otimes \mathcal R\right)\oplus \mathbb C\omega_0\oplus \mathbb C\omega_1$ and assume that $\chi_0\in\mathbb C$ and $\mathcal V$ as in \lemref{heisenbergprop}.   Then using \eqnref{c}, \eqnref{c*} and \lemref{rhorep}, the following defines a representation of the three point algebra ${\mathfrak g}$ on $\mathbb C[\mathbf x]\otimes \mathbb C[\mathbf y]\otimes \mathcal V$:
\begin{align*}
\tau(\omega_1)&=0, \qquad
\tau(\omega_0)=\chi_0=\kappa_0+4\delta_{r,0} ,  \\ 
\tau(f(z))&=-\alpha, \qquad
\tau(f^1(z))=- \alpha^1,   \\ \\
\tau(h(z))
&=2\left(:\alpha\alpha^*:+:\alpha^1\alpha^{1*}: \right)
     +\beta , \\  \\
\tau(h^1(z))
&=2\left(:\alpha^1\alpha^*: +(z^2+4z):\alpha\alpha^{1*}: \right) +\beta^1,  \\  \\
\tau(e(z)) 
&=:\alpha(\alpha^*)^2:+(z^2+4z):\alpha(\alpha^{1*})^2: +2 :\alpha^1\alpha^*\alpha^{1*}:+\beta\alpha^*+\beta^1\alpha^{1*} +\chi_0\partial\alpha^*  \\ \\
\tau(e^1(z))  
&=\alpha^1\alpha^*\alpha^* 
 	+(z^2+4z)\left(\alpha^1 (\alpha^{1*} )^2 +2 : \alpha \alpha^{*} \alpha^{1*}:\right)  \\
&\quad +\beta^1 \alpha^* +(z^2+4z)\beta \alpha^{1*} +\chi_0\left((z^2+4z)   \partial_z\alpha^{1*}   +(z+2)   \alpha^{1*} \right) .
\end{align*}
\end{thm}

Before we go through the proof it will be fruitful to review V. Kac's $\lambda$-notation (see \cite{MR99f:17033} section 2.2 and \cite{MR1873994} for some of its properties) used in operator product expansions.  If $a(z)$ and $b(w)$ are formal distributions, then
$$
[a(z),b(w)]=\sum_{j=0}^\infty \frac{(a_{(j)}b)(w)}{(z-w)^{j+1}}
$$
is transformed under the {\it formal Fourier transform} 
$$
F^{\lambda}_{z,w}a(z,w)=\text{Res}_ze^{\lambda(z-w)}a(z,w),
$$
 into the sum
\begin{equation*}
[a_\lambda b]=\sum_{j=0}^\infty \frac{\lambda^j}{j!}a_{(j)}b.
\end{equation*}
Set
$$
P(w) =w^2+4w.
$$
So for example we have the following 

\begin{lem} \label{pairs}
Given the definitions in the previous section, we have
\begin{enumerate}
\item $
[\beta^1_\lambda\beta^1]=-\left(2(w^2+4w)\lambda+(2w+4)\right)\kappa_0=-\left(2P\lambda+\partial  P\right)\kappa_0$ 
\item
$[:\alpha\alpha^*:{_\lambda}:\alpha\alpha^*:]=-\delta_{r,0}\partial\delta(z/w)$
\item
$
[:\alpha(\alpha^*)^2:{_\lambda}:\alpha(\alpha^*)^2:]=-4\delta_{r,0}:\alpha^*\partial \alpha^* :-4\delta_{r,0}:(\alpha^*)^2:\lambda
$.
\end{enumerate}
\end{lem}
Note that similar expressions hold for $\alpha^1$ and $\alpha^{1*}$.
\begin{proof} We'll prove (2) and (3). By Wick's Theorem
\begin{align*}
:\alpha(z)\alpha^*(z)::\alpha(w)\alpha^*(w):  &=:\alpha(z)\alpha^*(z)\alpha(w)\alpha^*(w): +\lfloor \alpha(z),\alpha^*(w)\rfloor :\alpha(w)\alpha^*(z):+\lfloor \alpha^*(z),\alpha(w)\rfloor :\alpha(z)\alpha^*(w): \\
&\quad +\lfloor \alpha(z),\alpha^*(w)\rfloor\lfloor \alpha^*(z),\alpha(w)\rfloor  \\
&= :\alpha(z)\alpha(w)\alpha^*(z)\alpha^*(w): + :\alpha(w)\alpha^*(z):\iota_{z,w}\left(\frac{1}{z-w}\right)+   :\alpha(z)\alpha^*(w):\iota_{z,w}\left(\frac{1}{w-z}\right)  \\
&\quad  -\delta_{r,0}\iota_{z,w}\left(\frac{1}{z-w}\right)^2
\end{align*}
 and
\begin{align*}
[: \alpha(z) \alpha^*(z)^2:,: \alpha(w) \alpha^*(w)^2:]&=2: \alpha(z) \alpha^*(z) \alpha^*(w)^2:\delta(z/w)-2: \alpha(w) \alpha^*(z)^2 \alpha^*(w):\delta(z/w) \\
&\quad -4\delta_{r,0}: \alpha^*(z) \alpha^*(w):\partial_w\delta(z/w)  \\ 
&= -4\delta_{r,0}: \alpha^*(z)\partial_w\left( \alpha^*(w):\delta(z/w)\right)+4\delta_{r,0}: \alpha^*(z)\partial_w \alpha^*(w):\delta(z/w)  \\
&= -4\delta_{r,0}:\partial_w\left( \alpha^*(w) \alpha^*(w):\delta(z/w)\right)+4\delta_{r,0}: \alpha^*(w)\partial_w \alpha^*(w):\delta(z/w)  \\
&= -4\delta_{r,0}:\partial_w \alpha^*(w) \alpha^*(w):\delta(z/w)-4\delta_{r,0}: \alpha^*(w) \alpha^*(w):\partial_w\delta(z/w).
\end{align*}

\end{proof}

\begin{proof} (Theorem \ref{mainthm})
We have need to check the following table is preserved under  $\tau$. 
\begin{table}[htdp]
\caption{}
\begin{center}
\begin{tabular}{c|cccccc}
$[\cdot_\lambda \cdot]$ & $f(w)$ & $f^1(w)$ & $h(w)$ & $h^1(w)$ & $e(w)$ & $e^1(w)$ \\
\hline
$f(z)$ & $0$ &  $0$ & $*$ &  $*$ &$*$  &$*$    \\
$f^1(z)$ &   &  $0$  & $*$  & $*$  & $*$  &$*$  \\
$h(z)$ &  &   &  $*$  & $*$  & $*$  & $*$ \\
$h^1(z)$ &&   &     & $*$  &    $*$ & $*$ \\
$e(z)$  &&   &     &  &    $0$ &  $0$  \\
$e^1(z)$ & &   &     &  &   &  $0$ \\
\end{tabular}
\end{center}
\label{default}
\end{table}%

Here $*$ indicates nonzero formal distributions that are obtained from the the defining relations \eqref{r1}, \eqref{r2}, and \eqref{r3}.
The proof is carried out using Wick's Theorem, Taylor's Theorem, and Lemma \ref{pairs} as one can see below:

\begin{align*}
[\tau(f)_\lambda\tau(f)]&=0,\quad [\tau(f)_\lambda\tau(f^1)]= 0,\\ \\
[\tau(f)_\lambda\tau(h)]&=-\Big[\alpha_\lambda\Big(2\left(\alpha\alpha^*+\alpha^1\alpha^{1*} \right)   +\beta \Big)\Big]=  -2\alpha =   2\tau(f) , \\  \\
[\tau(f)_\lambda\tau(h^1)]&=-[\alpha_\lambda\left(2\left(\alpha^1\alpha^*
   +P\alpha\alpha^{1*} \right)
     +\beta^1\right)  ]= -2\alpha^{1}=   2\tau(f^1), \\  \\
[\tau(f)_\lambda\tau(e)]&=- [\alpha_\lambda\left(:\alpha(\alpha^*)^2:+P:\alpha(\alpha^{1*})^2: +2 :\alpha^1\alpha^*\alpha^{1*}:+\beta\alpha^*+\beta^1\alpha^{1*} +\chi_0\partial\alpha^*\right)] \\  
&=-2\left(:\alpha\alpha^*:    + :\alpha^1\alpha^{1*}: \right)-\beta-\chi_0\lambda  
= -\tau(h)-\chi_0\lambda   \\ \\
\tau(f)_\lambda\tau(e^1)]&=- \Big[\alpha_\lambda\Big(\alpha^1(\alpha^*)^2 
 	+P\left(\alpha^1 (\alpha^{1*} )^2 +2 : \alpha \alpha^{*} \alpha^{1*}:\right)  \  +\beta^1 \alpha^* +P\beta \alpha^{1*} +\chi_0\left(P  \partial \alpha^{1*}   +\frac{1}{2}\partial  P   \alpha^{1*} \right) \Big)\Big]\\  
&=-2\left(:\alpha^1\alpha^*:    + P:\alpha\alpha^{1*}: \right)-\beta^1   =-\tau(h^1).
\\ \\ \\  \\
[\tau(f^1)_\lambda\tau(f^1)]&=0  \\ \\
[\tau(f^1)_\lambda\tau(h)]&=-[\alpha^1_\lambda\left(2\left(:\alpha\alpha^*:+:\alpha^1\alpha^{1*}: \right)   +\beta\right) ]=  -2\alpha^1 =   2\tau(f^1),  \\  \\
[\tau(f^1)_\lambda\tau(h^1)]&=-[\alpha^1_\lambda\left(2\left(:\alpha^1\alpha^*:
   +P:\alpha\alpha^{1*}: \right)
     +\beta^1\right)] =-2P\alpha^{1} =2P\tau(f^{1}) , \\  \\
[\tau(f^1)_\lambda\tau(e)]&=- [\alpha^1_\lambda\left(:\alpha(\alpha^*)^2:+P:\alpha(\alpha^{1*})^2: +2 :\alpha^1\alpha^*\alpha^{1*}:+\beta\alpha^*+\beta^1\alpha^{1*} +\chi_0\partial\alpha^*\right)] \\  
&=- \Big(  2P:\alpha\alpha^{1*}: +2 :\alpha^1\alpha^*: +\beta^1  \Big)  
=-\tau(h^1) \\ \\
[\tau(f^1)_\lambda\tau(e^1)]&=-[ \alpha^1_\lambda\Big(\alpha^1(\alpha^*)^2 
 	+P\left(\alpha^1 (\alpha^{1*} )^2 +2 : \alpha \alpha^{*} \alpha^{1*}:\right)  \  +\beta^1 \alpha^* +P\beta \alpha^{1*} +\chi_0\left(P  \partial \alpha^{1*}   +\frac{1}{2}(\partial  P)   \alpha^{1*} \right) \Big)] \\  
&=-\Big(
 	P\left(  2\left(:\alpha^1\alpha^{1*}:   + :\alpha\alpha^* :\right)
	+\beta+\chi_0\lambda  \right)+\frac{1}{2}\chi_0\partial P  \Big)  \\
&=-\left(P\tau(h)+P\chi_0\lambda+\chi_0\frac{1}{2}\partial P\right).
\end{align*}
 Note that $:a(z)b(z):$ and $:b(z)a(z):$ are usually not equal, but $:\alpha^1(w)\alpha^{*1}(w):=:\alpha^{1*}(w)\alpha^1(w):$ and $:\alpha(w)\alpha^*(w):=:\alpha^*(w)\alpha (w):$.  Thus we calculate

\begin{align*}
[\tau(h)_\lambda\tau(h)]&=\left(2\left(:\alpha\alpha^*:+:\alpha^1\alpha^{1*}: \right)   +\beta\right)  _\lambda \left(2\left(:\alpha\alpha^*:+:\alpha^1\alpha^{1*}: \right)   +\beta\right)]  \\
&= 4\Big(-:\alpha\alpha^*: +:\alpha^*\alpha: -:\alpha^1\alpha^{1*}: +:\alpha^{1*}\alpha^1:  \Big)  
-8\delta_{r,0}\lambda +[\beta_\lambda\beta]   \\
&=-2(4\delta_{r,0}+\kappa_0)\lambda
 \end{align*}
 Which can be put into the form of \eqref{r1}:
 \begin{align*}
\left[\tau(h(z)),\tau(h(w))\right]&=
 -2(4\delta_{r,0}+\kappa_0)\partial_{w}\delta(z/w)=
 -2\chi_0\partial_{w}\delta(z/w)=
 \tau\left(-2\omega_0\partial_{w}\delta(z/w)\right).  
  \end{align*}
  
Next we calculate

\begin{align*}
[\tau(h)_\lambda\tau(h^1)]&=\left[\left(2\left(:\alpha\alpha^*:+:\alpha^1\alpha^{1*}: \right)   +\beta \right) _\lambda \left(2\left(:\alpha^1\alpha^*:
   +P:\alpha\alpha^{1*}: \right)+\beta^1  \right) \right]\\  
&=4\Big(\left(:\alpha^* \alpha^1 : -:\alpha^1 \alpha^{*} :\right)   + P \left(-:\alpha\alpha^{1*} :  +:\alpha^{1*}\alpha :  \right)    \Big)  +[\beta_\lambda\beta^1]   
\end{align*}
Since $[a_n,a_m^{1*}]=[a^1_n,a_m^{*}]=0$, we have 
 \begin{align*}
\left[\tau(h(z)),\tau(h^1(w))\right]&=[\beta(z),\beta^1(w)]=0.
 \end{align*}
 As $\tau(\omega_1) =0$,  relation \eqref{r3} is satisfied. 
 
 We continue with

\begin{align*}
[\tau(h^1)_\lambda\tau(h^1)]&=\left[2\left(:\alpha^1\alpha^*:
   +P:\alpha\alpha^{1*}: \right)
     +\beta^1  \right)_\lambda \left(2\left(:\alpha^1 \alpha^* :
   +P:\alpha\alpha^{1*}: \right)
     +\beta^1 )  \right] \\  
&= 4P\left(-:\alpha\alpha^*:+:\alpha^{1*}\alpha^1:\right) +4P\left(-:\alpha^1\alpha^{1*} : +:\alpha^*\alpha: \right) 
  \\
&\quad  -8\delta_{r,0}P\lambda   -4\delta_{r,0}\partial P    + [\beta^1_\lambda \beta^1]   \\
&=-8\delta_{r,0}P\lambda   -4\delta_{r,0}\partial P   -2\kappa_0(P\lambda+\frac{1}{2}\partial P).
 \end{align*}
 Yielding the relation
\begin{align*}
\left[\tau(h^1(z)),\tau(h^1(w))\right]&=-2(4\delta_{r,0}+\kappa_0)\left(\left(w^2+4w\right)\partial_w\delta(z/w)   +(w+2))\delta(z/w)\right) \\
& = \tau(-(h,h)  \omega_0 P \partial_w\delta(z/w) -\frac{1}{2}(h,h) \partial P \omega_0  \delta(z/w))
\end{align*}

Next we calculate the $h$'s paired with the $e$'s:

\begin{align*}
[\tau(h)_\lambda \tau(e )]&=\Big[\Big(2\left(:\alpha\alpha^*:+:\alpha^1\alpha^{1*}: \right)
     +\beta \Big)_\lambda \\
& \quad :\alpha(\alpha^*)^2:+ P:\alpha(\alpha^{1*})^2: +2 :\alpha^1\alpha^*\alpha^{1*}:+ \beta\alpha^*+\beta^1\alpha^{1*} +\chi_0\partial\alpha^*\Big] \\ 
& = 4: \alpha (\alpha^*)^2:  - 2 : \alpha  (\alpha^*)^2:   - 4 \delta_{r,0} \alpha^* \lambda \\
&- 2P : \alpha(\alpha^{1*})^2:  
 + 4: \alpha^*  \alpha^1  \alpha^{1*}  :  \\
& + 2  \alpha^*  \beta 
 + 2\chi_0  \alpha^* \lambda +  2\chi_0  \partial  \alpha^*\\
 &+ 4P : \alpha (\alpha^{1*})^2: - 4 \delta_{r,0} \alpha^* \lambda \\
&  + 2     \beta^1   \alpha^{1*}  - 2 \lambda \alpha^* \kappa_0 \\
&= 2 \tau (e)
\end{align*}
and
\begin{align*}
[\tau(h^1)_\lambda \tau(e)] & = \left[ \left(2\left(:\alpha^1\alpha^*:
   +P:\alpha\alpha^{1*}: \right) + \beta^1 \right)_\lambda 
 :\alpha(\alpha^*)^2:+ P:\alpha(\alpha^{1*})^2: +2 :\alpha^1\alpha^*\alpha^{1*}:+ \beta\alpha^*+\beta^1\alpha^{1*} +\chi_0\partial\alpha^*\right]\\
 &=  -2  : \alpha^1( \alpha^*)^2 
 + 2P(2 : \alpha \alpha^* \alpha^{1*}: - :  \alpha^1  (\alpha^{1*})^2: - 2 \delta_{r,0} \alpha^{1*}  \lambda ) - 4 \delta_{r,0} \partial P  \alpha^{1*} \\
 &+ 4 :  \alpha^1( \alpha^*)^2: + 2 \alpha^* \beta^1 \\
& + 4P: \alpha \alpha^{1*} \alpha^* + 4P(  : \alpha^1 (\alpha^{1*})^2 : - : \alpha \alpha^{*}\alpha^{1*}:  -\delta_{r,0}   \alpha^{1*}\lambda ) \\
&+ 2P  \beta \alpha^{1*}+ 2 \chi_0 ( P\partial \alpha^{1*}   + \partial P \alpha^{1*}   +P \alpha^{1*}  
 \lambda  )\\
& - 2 ( P \lambda + \frac{1}{2} \partial P ) \alpha^{1*} \kappa_0\\
&=  2   :  \alpha^1( \alpha^*)^2:  +  2P  :  \alpha^1  (\alpha^{1*})^2:  + 4P:  \alpha \alpha^*  \alpha^{1*}:  +
2 \delta(z/w) \alpha^* \beta^1+ 2P\beta \alpha^{1*}      \\
& + 2P \chi_0 \partial   \alpha^{1*}  +  \partial P  \alpha^{1*} \chi_0 \\
& = 2\tau (e^1  )
\end{align*}

Next we must calculate

\begin{align*}
[\tau(h)_\lambda \tau(e^1)]&=\Big[\Big(2\left(:\alpha\alpha^*:+:\alpha^1\alpha^{1*}: \right)
     +\beta \Big)_\lambda \\
&\quad  : \alpha^1\alpha^*\alpha^* :
 	+P\left(:\alpha^1 (\alpha^{1*} )^2 :+2 : \alpha \alpha^{*} \alpha^{1*}:\right)  \\
&\quad +\beta^1 \alpha^* +P\beta \alpha^{1*}  +\chi_0\left(P  \partial \alpha^{1*}   +(1/2 \partial P )  \alpha^{1*} \right)\Big] \\  
&= 4  :\alpha^1 (\alpha^* )^2: - 4 P\delta_{r,0}   \alpha^{1*} \lambda \\ 
& + 2   \alpha^* \beta^1  - 2\delta(z/w) : \alpha^1 \alpha^* \alpha^* :\\
&-4P\delta_{r,0}  \alpha^{1*} \lambda + 2P  : \alpha^{1*}\alpha^1 \alpha^{1*}: \\
&+ 4P : \alpha^{1*} \alpha \alpha^*:
+2P\beta\alpha^{1*}    \\
&+ 2 \chi_0(   P \alpha^{1*}  \lambda +  P  \partial \alpha^{1*} + \frac{1} {2} \partial P   \alpha^{1*} ) 
  - 2 P \alpha^{1*}  \kappa_0 \lambda \\
& = 2:  \alpha^1( \alpha^*)^2:   +  2P : \alpha^1 (\alpha^{1*})^2:  + 4P: \alpha \alpha^* \alpha^{1*} : \\
 & + 2 \beta^1  \alpha^*   +2P\beta\alpha^{1*}     +2\chi_0\left(P  \partial_w\alpha^{1*}   +(1/2 \partial P )  \alpha^{1*} \right)\\
 & =  2\tau(e^1 )
\end{align*}
and the proof for $[\tau(h^1)_\lambda \tau(e^1)]$ is similar. 
%

We prove the Serre relation for just one of the relations, $[\tau(e)_\lambda \tau(e^1)]$, and the proof of the others ($[\tau(e)_\lambda \tau(e)]$, $[\tau(e^1)_\lambda \tau(e^1)]$) are similar as the reader can verify. 

\begin{align*} 
[\tau(e)_\lambda \tau(e^1)]
&=  \Big[:\alpha(\alpha^*)^2:  _\lambda \Big(\alpha^1(\alpha^*)^2 
 	+(w^2+4w)\left(\alpha^1 (\alpha^{1*} )^2 +2 : \alpha \alpha^{*} \alpha^{1*}:\right)  \\
&\qquad +\beta^1 \alpha^* +(w^2+4w)\beta \alpha^{1*}  +\chi_0\left((w^2+4w)   \partial_w\alpha^{1*}   +(w+2)   \alpha^{1*} \right)\Big)   \Big]\\  
&\quad + \Big[P:\alpha(\alpha^{1*})^2 :{_\lambda}  \Big(:\alpha^1(\alpha^*)^2 :
 	+(w^2+4w)\left(\alpha^1 (\alpha^{1*} )^2 +2 : \alpha \alpha^{*} \alpha^{1*}:\right)  \\
&\qquad +\beta^1 \alpha^* +(w^2+4w)\beta \alpha^{1*}  +\chi_0\left((w^2+4w)   \partial_w\alpha^{1*}   +(w+2)   \alpha^{1*} \right)\Big)   \Big]\\   
&\quad +   \Big[2 :\alpha^1\alpha^*\alpha^{1*}:{_\lambda}  \Big(:\alpha^1(\alpha^*)^2 :
 	+(w^2+4w)\left(\alpha^1 (\alpha^{1*} )^2 +2 : \alpha \alpha^{*} \alpha^{1*}:\right)  \\
&\qquad +\beta^1 \alpha^* +(w^2+4w)\beta \alpha^{1*}  +\chi_0\left((w^2+4w)   \partial_w\alpha^{1*}   +(w+2)   \alpha^{1*} \right)\Big)   \Big]\\   
&\quad    +  \Big[\beta\alpha^*{ _\lambda}  \Big(:\alpha^1(\alpha^*)^2 :
 	+(w^2+4w)\left(\alpha^1 (\alpha^{1*} )^2 +2 : \alpha \alpha^{*} \alpha^{1*}:\right)  \\
&\qquad +\beta^1 \alpha^* +(w^2+4w)\beta \alpha^{1*}  +\chi_0\left((w^2+4w)   \partial_w\alpha^{1*}   +(w+2)   \alpha^{1*} \right)\Big)    \Big]\\  
&\quad  +\Big[\beta^1\alpha^{1*}{ _\lambda}  \Big(:\alpha^1(\alpha^*)^2 :
 	+(w^2+4w)\left(\alpha^1 (\alpha^{1*} )^2 +2 : \alpha \alpha^{*} \alpha^{1*}:\right)  \\
&\qquad +\beta^1 \alpha^* +(w^2+4w)\beta \alpha^{1*}  +\chi_0\left((w^2+4w)   \partial_w\alpha^{1*}   +(w+2)   \alpha^{1*} \right)\Big)    \Big]\\   
&\quad +\Big[  \chi_0\partial\alpha^*{_\lambda}  \Big(:\alpha^1(\alpha^*)^2 :
 	+(w^2+4w)\left(\alpha^1 (\alpha^{1*} )^2 +2 : \alpha \alpha^{*} \alpha^{1*}:\right)  \\
&\qquad +\beta^1 \alpha^* +(w^2+4w)\beta \alpha^{1*}  +\chi_0\left((w^2+4w)   \partial_w\alpha^{1*}   +(w+2)   \alpha^{1*} \right)\Big)   \Big] \\ \\  
&=   \Big[:\alpha(\alpha^*)^2:  _\lambda \Big(:\alpha^1(\alpha^*)^2 :
 	+2P : \alpha \alpha^{*} \alpha^{1*}:  +\beta^1 \alpha^* \Big)   \Big]\\  \\ 
&\quad + \Big[P:\alpha(\alpha^{1*})^2 :{_\lambda}  \Big(:\alpha^1(\alpha^*)^2:
 	+P\left(:\alpha^1 (\alpha^{1*} )^2: +2 : \alpha \alpha^{*} \alpha^{1*}:\right) +\beta^1 \alpha^* \Big)   \Big]\\  \\ 
&\quad +   \Big[2 :\alpha^1\alpha^*\alpha^{1*}:{_\lambda}  \Big(\alpha^1
(\alpha^*)^2 
 	+P\left(:\alpha^1 (\alpha^{1*} )^2: +2 : \alpha \alpha^{*} \alpha^{1*}:\right)  +P\beta \alpha^{1*}  \\
&\hskip 100pt +\chi_0\left((w^2+4w)   \partial_w\alpha^{1*}   +(w+2)   \alpha^{1*} \right)\Big)   \Big]\\  \\ 
&\quad    +  \Big[\beta\alpha^*{ _\lambda}  \Big(2P : \alpha \alpha^{*} \alpha^{1*}: +\beta^1 \alpha^* +P\beta \alpha^{1*}\Big)    \Big]\\  \\
&\quad  +\Big[\beta^1\alpha^{1*}{ _\lambda}  \Big(:\alpha^1(\alpha^*)^2 :
 	+P\alpha^1 (\alpha^{1*} )^2  +\beta^1 \alpha^* +P\beta \alpha^{1*} \Big)    \Big]\\  \\  
&\quad 
	  +\Big[  \chi_0\partial\alpha^*{_\lambda}  \Big( 2P : \alpha \alpha^{*} \alpha^{1*}:\Big)   \Big]  
\end{align*}

\begin{align*}
&=   2:\alpha^1\alpha^*(\alpha^*)^2 :
 	+2P: \alpha( \alpha^{*})^2 \alpha^{1*}: -4P:\alpha( \alpha^{*})^2 \alpha^{1*}: -4\delta_{r,0}P:   \alpha^{*} \alpha^{1*}:\lambda-4\delta_{r,0}P:  \partial( \alpha^{*}) \alpha^{1*}:  \\
	&\hskip 100pt   +\beta^1 (\alpha^*)^2
	 \\  \\ 
&\quad -2P:\alpha(\alpha^*)^2\alpha^{1*}:+2P\alpha^1\alpha^*(\alpha^{1*})^2 \\ 
&\quad  -4\delta_{r,0}  P:\alpha^{1*} \alpha^*:\lambda-4\delta_{r,0}\partial P:\alpha^{1*} \alpha^*:-4\delta_{r,0} P:\partial\alpha^{1*} \alpha^*:
  \\
&\quad -2P^2:\alpha  (\alpha^{1*} )^3: +2 P^2: \alpha(\alpha^{1*})^3:  +P\beta^1 (\alpha^{1*})^2  
 \\  \\ 
&\quad -  2 :\alpha^1\alpha^*(\alpha^*)^2:
 	+ 4P:\alpha^1  \alpha^*(\alpha^{1*})^2:-2P:\alpha^1\alpha^* (\alpha^{1*} )^2:-4\delta_{r,0}P: \alpha^*\alpha^{1*}:\lambda  -4\delta_{r,0}P: \partial(\alpha^*)\alpha^{1*}: \\
&\quad+4 P : \alpha(\alpha^*)^2\alpha^{1*}:  -4 P :\alpha^1  \alpha^{*} (\alpha^{1*})^2:  \\
&\quad -4\delta_{r,0}  P: \alpha^*\alpha^{1*}:\lambda 
  -4\delta_{r,0}  P: \alpha^*\partial\alpha^{1*}:  + 2P\beta:\alpha^*\alpha^{1*}:\\  \\ 
&\quad +   2\chi_0\Big( P:\partial \alpha^*\alpha^{1*}:+P :\alpha^*\partial\alpha^{1*}:+P :\alpha^*\alpha^{1*}:\lambda+\frac{1}{2}(\partial P):\alpha^* \alpha^{1*}:  \Big)\\  \\ 
&\quad    - 2P \beta   \alpha^{*} \alpha^{1*}
-2\kappa_0P\alpha^* \alpha^{1*} \lambda -2\kappa_0P\partial \alpha^* \alpha^{1*}   \\  \\
&\quad  -\beta^1(\alpha^*)^2 
 	-P\beta^1 (\alpha^{1*} )^2  -\kappa_0\left(2P\alpha^*\alpha^{1*}\lambda+2P\alpha^* \partial\alpha^{1*}+\partial  P\alpha^*\alpha^{1*} \right)    \\
&\quad 
+2 \chi_0P \alpha^{*} \alpha^{1*}\lambda     \\ \\ 
&=   -4\delta_{r,0}P:   \alpha^{*} \alpha^{1*}:\lambda-4\delta_{r,0}P:  \partial( \alpha^{*}) \alpha^{1*}:  \\
	&\hskip 100pt  +\chi_1\left(2:\alpha^*\partial \alpha^*:+:(\alpha^*)^2:\lambda  \right) \\  \\ 
&\quad  -4\delta_{r,0}  P:\alpha^{1*} \alpha^*:\lambda-4\delta_{r,0}\partial P:\alpha^{1*} \alpha^*:-4\delta_{r,0} P:\partial\alpha^{1*} \alpha^*:
  \\
&\quad
-4\delta_{r,0}P: \alpha^*\alpha^{1*}:\lambda  -4\delta_{r,0}P: \partial(\alpha^*)\alpha^{1*}: \\
&\quad -4\delta_{r,0}  P: \alpha^*\alpha^{1*}:\lambda 
    -4\delta_{r,0}  P: \alpha^*\partial\alpha^{1*}:  \\  \\ 
&\quad +   2\chi_0\Big( P:\partial \alpha^*\alpha^{1*}:+P :\alpha^*\partial\alpha^{1*}:+P :\alpha^*\alpha^{1*}:\lambda+\frac{1}{2}(\partial P):\alpha^* \alpha^{1*}:  \Big)\\  \\ 
&\quad   
-\kappa_0P\alpha^* \alpha^{1*} \lambda -\kappa_0P\partial \alpha^* \alpha^{1*}   \\  \\
&\quad  -\kappa_0\left(2P\alpha^*\alpha^{1*}\lambda+2P\alpha^* \partial\alpha^{1*}+\partial  P\alpha^*\alpha^{1*} \right)    \\
&\quad
 -2 \chi_0P \alpha^{*} \alpha^{1*}\lambda  \\
 &=0.
\end{align*}

\end{proof}

\section{Further Comments}
 We plan to use the above construction to help elucidate the structure of these representations of a three point algebra, describe the space of their intertwining operators and eventually describe the center of a certain completion of the universal enveloping algebra for the three point algebra.


\bibliographystyle{amsplain}
\def\cprime{$'$} \def\cprime{$'$} \def\cprime{$'$}
\providecommand{\bysame}{\leavevmode\hbox to3em{\hrulefill}\thinspace}
\providecommand{\MR}{\relax\ifhmode\unskip\space\fi MR }
\providecommand{\MRhref}[2]{%
  \href{http://www.ams.org/mathscinet-getitem?mr=#1}{#2}
}
\providecommand{\href}[2]{#2}

 \end{document}